\documentclass[12pt,reqno]{amsart}
\usepackage[utf8]{inputenc}
\usepackage[letterpaper, margin=1in]{geometry}

\usepackage[utf8]{inputenc}
\usepackage{amsmath, amsthm, amssymb, color, mathbbol}
\usepackage[shortlabels]{enumitem}
\usepackage{graphicx}
\usepackage{makecell}
\usepackage[ruled]{algorithm2e}
\usepackage{multicol}
\usepackage{array,amscd}
\usepackage{amsfonts}
\usepackage[T1]{fontenc} 
\usepackage{appendix}
\usepackage[arrow, matrix, curve]{xy}
\usepackage[square,comma,sort,numbers]{natbib}
\usepackage{booktabs}
\usepackage{siunitx}
\usepackage{multirow}
\usepackage[font=small,labelfont=bf]{caption}
\usepackage{subfig}
\usepackage{graphicx, array,amscd}
\usepackage{amsfonts}
\usepackage{bbm}
\usepackage[T1]{fontenc}
\usepackage[square,comma,sort,numbers]{natbib}
\usepackage{booktabs}
\usepackage{blkarray}

\usepackage[dvipsnames]{xcolor}
\usepackage[pdftex]{hyperref}
\hypersetup{
    colorlinks=true,
    linkcolor={MidnightBlue},
    citecolor={magenta},
    urlcolor={MidnightBlue}
}

\usepackage{tikz}
\usetikzlibrary{snakes}
\usetikzlibrary{shapes}
\usetikzlibrary{arrows}
\usetikzlibrary{positioning}
\usetikzlibrary{calc}
\usetikzlibrary{shapes.geometric}

\setenumerate{leftmargin=*}
\allowdisplaybreaks[3]	

\theoremstyle{plain}
\newtheorem{theorem}{Theorem}[section]
\newtheorem{proposition}[theorem]{Proposition}
\newtheorem{corollary}[theorem]{Corollary}
\newtheorem{conjecture}[theorem]{Conjecture}
\newtheorem{lemma}[theorem]{Lemma}

\newtheorem*{thm*}{Theorem}
\newtheorem*{lemma*}{Lemma}
\newtheorem*{prop*}{Proposition}
\newtheorem*{cor*}{Corollary}
\newtheorem*{conj*}{Conjecture}
\newtheorem*{alg*}{Algorithm}

\theoremstyle{definition}
\newtheorem{definition}[theorem]{Definition}
\newtheorem{example}[theorem]{Example}
\newtheorem{remark}[theorem]{Remark}

\newcommand{\RR}{\mathbb{R}}

\newcommand{\cm}{\mathcal{M}}

\definecolor{darkgreen}{rgb}{0,0.4,0}
\definecolor{MyBlue}{rgb}{0,0.08,0.7} 
\definecolor{MyRed}{rgb}{0.85,0.08,0} 

\DeclareMathOperator{\pa}{pa}
\DeclareMathOperator{\an}{an}
\DeclareMathOperator{\ch}{ch}
\DeclareMathOperator{\rank}{rank}
\DeclareMathOperator{\neighbors}{ne}
\def\ci{\perp\kern-1.3ex\perp}
\def\nci{\not\kern-0.3ex\ci}

\definecolor{benpurple}{RGB}{180, 0, 240}

\makeatletter
\renewcommand*\env@matrix[1][\arraystretch]{%
  \edef\arraystretch{#1}%
  \hskip -\arraycolsep
  \let\@ifnextchar\new@ifnextchar
  \array{*\c@MaxMatrixCols c}}
\makeatother

\title[Identifiability of linear SEMs using Algebraic Matroids]{Identifiability of Homoscedastic Linear Structural Equation Models using Algebraic Matroids}
\author{Mathias Drton, Benjamin Hollering, and Jun Wu}
\date{August 4, 2023}

\begin{document}

\begin{abstract}
We consider structural equation models (SEMs), in which every variable is a function of a subset of the other variables and a stochastic error. Each such SEM is naturally associated with a directed graph describing the relationships between variables. When the errors are homoscedastic, recent work has proposed methods for inferring the graph from observational data under the assumption that the graph is acyclic (i.e., the SEM is recursive). In this work, we study the setting of homoscedastic errors but allow the graph to be cyclic (i.e., the SEM to be non-recursive). Using an algebraic approach that compares matroids derived from the parameterizations of the models, we derive sufficient conditions for when two simple directed graphs generate different distributions generically. Based on these conditions, we exhibit subclasses of graphs that allow for directed cycles, yet are generically identifiable. We also conjecture a strengthening of our graphical criterion which can be used to distinguish many more non-complete graphs.  
\end{abstract}

\maketitle

\section{Introduction}
Graphical models provide a powerful tool for modeling stochastic dependence between random variables. A random vector $X=(X_1,\dots, X_p)$ is represented by a graph, in which the vertices encode random variables and the directed edges encode causal relationships. In this paper we study structural equation models where the relationships between the random variables are specified by a system of equations which comes from the underlying graph.

Structure learning from (observational) data is a fundamental problem in the area of graphical models. For this problem to be well-defined, structural identifiability results are needed. A well-known result is that directed acyclic graphical (DAG) models can be identified up to Markov equivalence class \cite{pearl:2009,spirtes:2000}, under the assumption of faithfulness and causal sufficiency. 
Many works have discussed the structural identifiability under stronger distributional assumptions, e.g., linear non-Gaussian acyclic models are identifiable \cite{Shimizu06alinear}. These identifiability results have been  extended to some functional model classes of nonlinear additive noise models in \cite{NIPS2008_f7664060,Peters11identifiabilityof}. Moreover, some recent works \cite{NIPS2015_fccb60fb,pmlr-v89-park19a} discuss structural identifiability in models where the conditional distribution of every node given its parents is discrete. Another type of distributional assumption is that the model is \emph{homoscedastic} which means that the stochastic error of each variable has the same variance. This assumption is reasonable and also natural when it comes to variables from similar domains. With equal error variance, the structural identifiability of linear DAG models has been proved in \cite{ideneqPeters,eqvar/biomet/asz049}.

In many applications, feedback and self-regulations among variables may exist and thus it is more suitable to apply cyclic models rather than the DAG models. However, learning directed cyclic graphs is also much more difficult because of the existence of feedback loops. The Markov equivalence class of directed cyclic graphs was first studied \cite{Richardson96Algorithm,Richardson96Equivalence} and the cyclic causal discovery algorithm was proposed for estimating the Markov equivalence class.

Recently weaker assumptions on sparsity are considered \cite{park2016identifiability}, in the sense of number of edges or number of d-separations. When it comes to recovering the exact graphical structure from observational data, additional distributional assumptions or special model restrictions are required. An example is that, \cite{NIPS2011_idenbivarite} proves the structural identifiability of bivariate Gaussian cyclic graphical models with additive Gaussian noise.

We address the problem of structural identifiability of homoscedastic linear structural equation models (SEM) with directed cyclic graphs 
through the lens of \emph{algebraic statistics} \cite{ASseth}. This means that instead of focusing on the statistical model, we study the naturally associated algebraic variety obtained by taking the Zariski closure of the model. Algebraic and combinatorial techniques have been very successful in proving identifiability results, particularly in phylogenetics where one can reduce the problem to a finite number of cases \cite{gross2018distinguishing, sullivant2012disentangling} or use the inherent structure of the model to certify that the parameters are identifiable \cite{allman2010identifiability}. Often, one can also use to algebraic structure to provide a graphical criterion for identifiability such as the well known half-trek criterion of \cite{foygel2012halftrek} for linear SEMs.  

In this paper, we focus on the algebraic matroid associated to homoscedastic linear SEMs which is a represented by the Jacobian of the parameterization of the model. This allows us to find sufficient graphical conditions for identifiability and avoids time-consuming Gr\"obner basis calculations for the vanishing ideals of the models. We use the structure of the algebraic matroid to obtain several graphical criteria for identifiability which can be easily checked. This approach for proving identifiability results was first developed by \cite{Seth19M} based on the characterization of algebraic matroids presented in \cite{rosen2014computing}.

The remainder of this paper is structured as follows. Section \ref{sec:pre} provides some background on linear structural equation models, generic identifiability, and algebraic matroids. In Section \ref{sec:jac} we describe the general structure of Jacobian matroid of a linear SEMs on a directed graphs. In Section \ref{sec:DistinguishingWithMatroids} we provide several sufficient graphical conditions for identifiability by showing that they correspond to the associated statistical models having different matroids. This allows us to exhibit some families of graphs which are generically identifiable. Finally, in Section \ref{sec:Conjectures} we discuss some conjectures and provide some evidence that our graphical criteria might be able to distinguish many more non-complete graphs.

\section{Preliminaries}
\label{sec:pre}
In this section we provide some background on linear structural equation models (SEMs) which will be our primary interest throughout this paper. We then provide some background on generic identifiability and matroids which we will be the main tool we use to prove identifiability results for linear SEMs. 

\subsection{Linear structural equation models}
Let $\varepsilon=(\varepsilon_i:i\in V)$ be a vector of random errors and $X=(X_i:i\in V)$ be a random vector satisfying the structural equation system:
$$
X=\Lambda^TX+\varepsilon,
$$
in which $\Lambda=(\lambda_{ij})\in\mathbb{R}^{V\times V}$ with the unknown coefficient $\lambda_{ij} (i\neq j)$ being the direct effect of $X_j$ on $X_i$. Suppose that $\varepsilon$ has postive definite covariance matrix $\Omega=(\omega_{ij})\in \mathbb{R}^{V\times V}$. Let $I$ be the identity matrix. When $I-\Lambda$ is invertible, the equation system has a unique solution $X=(I-\Lambda)^{-T}\varepsilon$, with covariance matrix 
$${\rm Var}[X]=\Sigma=(I-\Lambda)^{-T}\Omega(I-\Lambda)^{-1}.$$ 

The linear structural equation model can be naturally represented by a graph. For independent errors $\varepsilon$ (i.e. $\Omega$ is diagonal), the linear SEM is associated to a directed graph $G=(V,D)$, where $V$ is the set of nodes and $D\subseteq V\times V$ is the edge set. Every node in $V$ represents a random variable. Elements in $D$ are ordered pairs $(i,j),\ i\neq j$, also denoted by $i\rightarrow j$, encoding the causal relationships between random variables. 

In a directed graph $G=(V,D)$, if  $(i,j)\in D$, we say that $i$ is a \emph{parent} of $j$ and $j$ is a \emph{child} of $i$, denoted by $i\in \pa(j)$ and $j\in \ch(i)$. A node $i$ is a \emph{ancestor} of $j$, denoted by $i \in \an(j)$, if there exists a directed path from $i$ to $j$. If there exists an edge between $i$ and $j$, we say that $i$ and $j$ are \emph{adjacent}. A \emph{collider triple} in $G$ is a triple of nodes $(i,j,k)$ having substructure $i\rightarrow j \leftarrow k$. The pointed node $j$ is a \emph{collider}. When the nodes pointing into a collider are not adjacent, we say that the collider is \emph{unshielded} and the tail nodes are \emph{adjacent via collider}. The \emph{skeleton} of $G$ is the undirected graph obtained by replacing all edges with undirected edges. 

Throughout this paper we will assume that the graph is simple, and 
the errors $\varepsilon_i$ are independent Gaussians with equal variance $\omega_{ii}=\omega$. That is, the edge weights $\lambda_{ij}$ and $\lambda_{ji}$ cannot be both nonzero for $i\neq j$, and $\Omega$ is a multiple of identity matrix $I$. With these assumptions it is often easier to consider the precision matrix $K = \Sigma^{-1}$ rather than $\Sigma$ itself. In particular, $K$ is given by the equation
\[
K=\Sigma^{-1}=(I-\Lambda)(\omega I)^{-1}(I-\Lambda)^{T}:=s(I-\Lambda)(I-\Lambda)^{T},\ s=\frac{1}{\omega}
\]
and we denote the corresponding map $(\Lambda,s)\mapsto K$ by $\psi(\Lambda,s)$.

Let $\mathbb{R}^D$ be the set of real $V\times V$ matrices $\Lambda=(\lambda_{ij})$ with support in $D$, i.e.,
$$
\mathbb{R}^D:=\{\Lambda\in V\times V:\ \Lambda_{ij}=0 \text{ if } i\rightarrow j\notin D\}.
$$
We also define $\mathbb{R}^D_{\text{reg}}$ to be the subset of matrices $\Lambda\in\mathbb{R}^D$ for which $I-\Lambda$ is invertible, since this is required for the associated linear SEM to be well-defined. 
\begin{definition}
The linear Gaussian precision model given by directed graph $G=(V,D)$, is the family of all multivariate normal distribution on $\mathbb{R}^V$ which is parameterized by the map
$$
M_G=\left\{K: K=\psi_G(\Lambda,s),\ \Lambda\in \mathbb{R}^D_{\text{\rm reg}} \text{ \rm and } s\in \mathbb{R}^+\right\}.
$$
The precision matrix parameterization of the model is the map
\begin{align*}
\psi_G: \mathbb{R}^D\times \mathbb{R}^+ &\mapsto PD_V,\\
(\Lambda,s) &\mapsto s(I-\Lambda)(I-\Lambda)^T,
\end{align*}
in which $PD_V$ is the cone of positive definite symmetric $V\times V$ matrices.
\end{definition}

\begin{example}
\label{ex:linearCyclicSEM}
Consider the following 4-node directed graph and the linear structural equation model which is given recursively on the left. 
\vspace{-1em}
\begin{multicols}{2}
\begin{align*}
X_1 &= \varepsilon_1\\
X_2 &= \lambda_{12}X_1 + \lambda_{42}X_4 + \varepsilon_2\\
X_3 &= \lambda_{23}X_2 + \varepsilon_3\\
X_4 &= \lambda_{34}X_3 + \varepsilon_4
\end{align*}
\mbox{ }

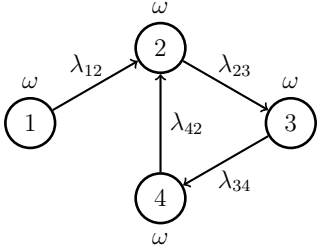
\scalebox{0.8}{
\begin{minipage}[t]{\textwidth}
\begin{tikzpicture}
[main_node/.style={circle,fill=white,draw,minimum size=2 em,inner sep=2pt, line width=1.25pt]}]

\node[main_node, label=above:{$\omega$}] (1) at (0, 0) {$1$};
\node[main_node, label=above:{$\omega$}] (2) at (2.165, 1.25) {$2$};
\node[main_node, label=above:{$\omega$}] (3) at (4.33, 0) {$3$};
\node[main_node, label=below:{$\omega$}] (4) at (2.165, -1.25) {$4$};

\draw [->, line width=1pt] (1) -- node[left=1.5mm, above] {$\lambda_{12}$} (2);
\draw [->, line width=1pt] (2) -- node[right=1.5mm, above] {$\lambda_{23}$} (3);
\draw [->, line width=1pt] (3) -- node[right=1.5mm, below] {$\lambda_{34}$} (4);
\draw [->, line width=1pt] (4) -- node[right] {$\lambda_{42}$} (2);

\end{tikzpicture}
\end{minipage}
}
\end{multicols}
\vspace{-1em}
All errors $\varepsilon_i$ follows the same normal distribution $\mathcal{N}(0,\omega)$ so the associated parameters of the model are $(\Lambda, s)$ which have the form
$$
{\Lambda}=
\begin{pmatrix}
0 & \lambda_{12} & 0 & 0\\
0 & 0 & \lambda_{23} & 0\\
0 & 0 & 0 & \lambda_{34}\\
0 & \lambda_{42} & 0 & 0
\end{pmatrix}
, \qquad
s=\frac{1}{\omega}. 
$$
Then the model $M_G$ associated to the graph $G$ above consists of precision matrices $K$ of the form
$$
K=s(I-\Lambda)(I-\Lambda)^T=
\begin{pmatrix}
s(1+\lambda_{12}^2) & -s\lambda_{12} & 0 & s\lambda_{12}\lambda_{42}\\
-s\lambda_{12} & s(1+\lambda_{23}^2) & -s\lambda_{23} & -s\lambda_{42}\\
0 & -s\lambda_{23} & s(1+\lambda_{34}^2) & -s\lambda_{34}\\
s\lambda_{12}\lambda_{42} & -s\lambda_{42} & -s\lambda_{34} & s(1+\lambda_{42}^2)
\end{pmatrix}.
$$
\end{example}

\subsection{Generic identifiability}
Suppose we have a family of models $\{M_{{i}}\}^k_{i=1}$ corresponding to graphs $G_i$ all of which have the same node set $V$ and thus all models sit in the same cone $PD_V$. 
If $M_{{i_1}}\cap M_{{i_2}}=\emptyset$ for each distinct pair $(i_1,i_2)$, 
we say that the discrete parameter $i$, or the graph collection $\{G_i\}_{i=1}^k$ is \emph{globally identifiable}. This requirement is typically too restrictive and cannot be fulfilled in many cases. The following weaker notion of identifiability is often used instead. 

\begin{definition}
Let $\{M_{{i}}\}^k_{i=1}$ be a finite set of algebraic models which sit in the same ambient space, the discrete parameter $i$ is \emph{generically identifiable} if for each pair of $(i_1,i_2)$, 
$$
\dim(M_{{i_1}}\cap M_{{i_2}})<\min\left(\dim(M_{{i_1}}),\dim(M_{{i_2}})\right).
$$
\end{definition}

Generic identifiability is sometimes called \emph{model distinguishability} \cite[Section 16]{ASseth} and is particularly nice since algebraic and geometric tools may be used to study these types of identifiability questions. Geometrically, this condition ensures that the intersection of any two models in the family is a Lebesgue measure zero subset of both models. However, this definition of generic identifiability is not appropriate for our setting since it implies that two graphs $G_1$ and $G_2$ with corresponding models $M_1$ and $M_2$ such that $M_1 \subseteq M_2$
are not generically identifiable, which is not desired. We instead refer to the notion of \emph{quasi equivalence} in \cite{CDE19,NEURIPS2020spar}, which requires the intersection of two models has nonzero measure under the Lebesgue measure defined over the \textbf{union} of both models. This leads to the following definition of generic identifiability for directed graph structures. 
\begin{definition}
Let $\{M_{{i}}\}^k_{i=1}$ be a finite set of algebraic models which sit in the cone $PD_V$, the parameter $i$ (or the model family) is \emph{generically identifiable} if for each pair of $(i_1,i_2)$, 
$$
\dim(M_{{i_1}}\cap M_{{i_2}})<\max\left(\dim(M_{{i_1}}),\dim(M_{{i_2}})\right).
$$
\end{definition}
This definition immediately implies that two models with different dimensions are generically identifiable. This means we can focus on the identifiability of models of the same dimension. When $\dim(M_1)=\dim(M_2)$, the $\min$ and $\max$ functions are actually the same. Two irreducible models of the same dimension must either be equal or have lower dimensional intersections.


\subsection{Proving Identifiability Results with Matroids}\label{sec2.3}
Since the linear SEMs we study are algebraic, the \emph{vanishing ideal} $\mathcal{I}(M)$ of the model $M$ which is
\[
\mathcal{I}(M)=\{f\in\mathbb{R}[x]:\ f(x)=0 \text{ for all } x\in M \} 
\]
can be used to answer many questions concerning the model. The following well-known proposition illustrates how vanishing ideals may be used to certify generic identifiability. 
\begin{proposition}\label{polyid}
\cite[Proposition 16.1.12]{ASseth} 
Let $M_1$ and $M_2$ be two irreducible algebraic models (e.g. parameterized models) which sit inside the same ambient space. If there exist polynomials $f_1$ and $f_2$ such that
$$
f_1\in\mathcal{I}(M_1)\ \backslash\ \mathcal{I}(M_2) \text{ and } f_2\in\mathcal{I}(M_2)\ \backslash\ \mathcal{I}(M_1) 
$$
then $\dim(M_1\cap M_2) <\min(\dim(M_1),\dim(M_2))$.
\end{proposition}
While this proposition can be a strong tool and has been successfully used to prove many identifiability results \cite{allman2010identifiability, gross2021distinguishing}, it typically requires many expensive Gr\"obner basis computations which become untenable as the size of the models grow. In this paper we will instead follow the approach developed in \cite{Seth19M}
which uses \emph{algebraic matroids} as an alternative to vanishing ideals to prove identifiability results. We summarize the key points here. For more details on matroids we refer the reader to \cite{oxleyMatroidTheory92}.

\begin{definition}
A \emph{matroid} $\mathcal{M}=(E,\mathcal{I})$ is a pair where $E$ is a finite set and $\mathcal{I}\subseteq 2^E$ satisfies
\begin{enumerate}[{\rm (1)}]
\item $\emptyset\in \mathcal{I}$
\item If $I'\subseteq I \in \mathcal{I}$, then $I'\in\mathcal{I}$
\item If $I_1, I_2\in\mathcal{I}$ and $|I_2|>|I_1|$, then there exists $e\in I_2\backslash I_1$ such that $I_1\cup e\in\mathcal{I}$
\end{enumerate}
\end{definition}

\begin{example}[Linear Matroid]
Let $E = \{1, 2, 3, 4\}$ be the ground set and let $A \in \RR^{3 \times 4}$ be the matrix
\[
A =
\begin{pmatrix}
1 & 1 & 0 & -2 \\
2 & 0 & 1 & -4 \\
1 & 1 & 0 & -2
\end{pmatrix}.
\]
Then we define a matroid $\mathcal{M}$ with independent sets $I \subseteq \{1,2,3,4\}$ such that the submatrix $A_I$ consisting of only the columns indexed  by $i \in I$ has full rank. In other words, the independent sets of $\mathcal{M}$ correspond to subsets of the columns of $A$ that are linearly independent over $\RR$.

For example the set $\{2, 3\}$ is clearly an independent set of $\cm$ since the second and third column of $A$ are linearly independent however the set $\{1,2,3\}$ is not an independent set. 
\end{example}

More generally, any linear space or matrix $A \in k^{m \times n}$ over a field $k$ defines a matroid on the set $\{1, \ldots n \}$ as is illustrated in the above example. The following definition demonstrates that this idea is not limited to linear spaces. 

\begin{definition}
Let $W\subset k^n$ be an irreducible variety over the field $k$ and for $S\subseteq [n]$ let $\pi_S:k^n\rightarrow k^{|S|}$ be the projection onto the coordinates in $S$. Let $\overline{\pi_S(W)}$ be the Zariski closure of the projection of $W$. Then the pair $([n], \mathcal{I}_W)$ defines a matroid where
$$
\mathcal{I}_W=\{S\subseteq[n]:\overline{\pi_S(W)}=k^{|S|}\},
$$
which is called the \emph{coordinate projection matroid} of $W$ and denoted by $\mathcal{M}(W)$.
\end{definition}

\begin{proposition}\label{coorid}
\cite[Proposition 3.1]{Seth19M}
Let $M_1$ and $M_2$ be two irreducible algebraic models sit in the same ambient space. Without loss of generality assume that $\dim(M_1)\geq\dim(M_2)$. If there exists a subset $S$ of the coordinates such that
$$
S\in \mathcal{M}(\overline{M_2})\ \backslash\ \mathcal{M}(\overline{M_1}),
$$
then $\dim(M_1\cap M_2)<\min(\dim(M_1),\dim(M_2))$.
\end{proposition}

In our settings, the model $M$ (and variety $\overline{M}$) is parameterized and thus irreducible. Under this condition, there exists another equivalent representation of the coordinate projection matroid $\mathcal{M}(W)$. Since we focus on the case of $\dim(M_1)=\dim(M_2)$, the roles of $M_1$ and $M_2$ are indeed symmetric. Either $S\in \mathcal{M}(\overline{M_2})\ \backslash\ \mathcal{M}(\overline{M_1})$ or $S\in \mathcal{M}(\overline{M_1})\ \backslash\ \mathcal{M}(\overline{M_2})$ implies generic identifiability.



\begin{proposition}\label{jacmat}
\cite{rosen2014computing}
Suppose that $\phi(\theta_1,...,\theta_d)=(\phi_1(\theta),...,\phi_n(\theta))$ parameterizes $W$ (i.e. $W=\overline{\phi(k^d)}$). Let
$$
J(\phi)=\left(\frac{\partial\phi_j}{\partial\theta_i}\right),\ 1\leq i\leq d,\ 1\leq j\leq n
$$
be the transpose of the Jacobian matrix of $\phi$. Then the matroid defined by the columns of the matrix $J(\phi)$ using linear independence over the fraction field $Frac(k[\theta])=k(\theta)$ gives the same matroid as the coordinate projection matroid $\mathcal{M}(\overline{\phi(k^d)})$. We call it the \textbf{Jacobian matroid} of $W$.
\end{proposition}

The statement of Proposition \ref{coorid} is the same as that Proposition 3.1 in \cite{Seth19M} except we have replaced the probability simplex with any ambient space. This proposition requires a stronger condition than Proposition \ref{polyid} but is much easier to check with the help of Proposition \ref{jacmat} since checking whether a set $S$ is in independent in the matroid $\mathcal{M}(\overline{M_i})$ amounts to computing the rank of the submatrix $J(\psi_i)_S$. This proposition will be our main tool to prove identifiability throughout the rest of this paper.

\section{Jacobian Structure of Gaussian Structural Equation Models} 
\label{sec:jac}
In this section we provide a combinatorial characterization of the entries of the Jacobian matrix of linear Gaussian precision SEMs. We then show that after appropriate row transformations, the entries of the Jacobain are all either constant or a single variable. We end this section with a few necessary graphical conditions which are required for two graphical models to have the same matroid. 

Throughout this paper we use the term ``Jacobian'' to refer to the transpose of the usual Jacobian matrix. The (transposed) Jacobian matrix, denoted by $J(\psi_G)$, is of size $(|D|+1)\times \frac{|V||V+1|}{2}$.  Each column of $J(\psi_G)$ corresponds to one entry $K_{ij}\ (i\leq j)$ of the precision matrix and each row corresponds to one edge weight $\lambda_{kl}$ or the inverse of common error variance $s$. For simplicity we will write $J$ for the Jacobian $J(\psi_G)$, if there is no ambiguity of indexing. When there are several Jacobian matrices corresponding to different graphs, we distinguish them by superscripts. The following lemma provides an explicit formula for the entries of $J$. 

\begin{lemma}\label{J}
Let $G = (V, D)$ be a directed graph. Then the entries of $J = J(\psi_G)$ are given by:
\begin{enumerate}[(1)]
\item For $i\in V$ and $(k,l)\in D$,
$$
J_{\lambda_{kl},K_{ii}}=
\begin{cases}
2s\lambda_{il}, &\quad k=i,\\
0, &\quad \text{else}.
\end{cases}
$$
\item
For $i,j\in V, i\neq j$ and $(k,l)\in D$,
$$
J_{\lambda_{kl},K_{ij}}=
\begin{cases}
-s, &\quad \{k,l\}=\{i,j\},\\
s\lambda_{jl},&\quad k=i \text{ and } (j,l)\in D,\\
s\lambda_{il},&\quad k=j \text{ and } (i,l)\in D,\\
0,&\quad \text{else}.
\end{cases}
$$
\item The partial derivatives w.r.t. the inverse error variance $s$ are
\begin{align*}
J_{s,K_{ii}}&=\left(1+\sum_{l \in {\ch}(i)} \lambda_{il}^2\right),\\
J_{s,K_{ij}}&=\left(-\lambda_{ij}1_{\{(i,j)\in D\}}-\lambda_{ji}1_{\{(j,i)\in D\}}+\sum_{l \in { \ch}(i)\cap{\ch}(j)} \lambda_{il}\lambda_{jl}\right).
\end{align*}
\end{enumerate}
\end{lemma}
Observe that this lemma clearly indicates the nonzero pattern of the $J(\psi_G)$. In particular, the columns of $K_{ii}$ have nonzero entries in the rows corresponding to outgoing edges of $i$ and the row for $s$. In the columns $K_{ij},\ i\neq j$, the entry in the row $\lambda_{ij}$ (or $\lambda_{ji}$) is nonzero if $(i,j)\in D$ (or $(j,i)\in D$); other nonzero entries are in the row pairs with those edges pointing to a common child of $i$ and $j$. The following example illustrates this lemma. 

\begin{example}\label{ex_jac}
Consider the graph $G=(V,D)$ pictured in Figure \ref{fig:diamond}, where $V=\{1,2,3,4\}$ and $D=\{(1,2),(2,4),(1,3),(3,4)\}$. The Jacobian $J(\psi)$ is given by
$$\begin{blockarray}{ccccccccccc}
K_{11} & K_{22} & K_{33} & K_{44} & K_{12} & K_{23} & K_{34} & K_{13} & K_{24} & K_{14} \\
\begin{block}{(cccccccccc)c}
  2s\lambda_{12} & 0 & 0 & 0    & -s & 0 & 0 & 0 & 0 & 0 & \lambda_{12}\\
  2s\lambda_{13} & 0 & 0 & 0    & 0 & 0 & 0 & -s & 0 & 0 & \lambda_{13}\\
  0 & 2s\lambda_{24} & 0 & 0    & 0 & s\lambda_{34} & 0 & 0 & -s & 0 & \lambda_{24}\\
  0 & 0 & 2s\lambda_{34} & 0    & 0 & s\lambda_{24} & -s & 0 & 0 & 0 & \lambda_{34}\\
  1+\lambda_{12}^2+\lambda_{13}^2 & 1+\lambda_{24}^2 & 1+\lambda_{34}^2 & 1 & -\lambda_{12} & \lambda_{24}\lambda_{34} & -\lambda_{34} & -\lambda_{13} & -\lambda_{24} & 0 & s\\
\end{block}
\end{blockarray}.
$$
\end{example}

\begin{lemma}
\label{lemma:simplifiedSRowJac}
Let $G = (V, D)$ be a graph and $J$ be the corresponding Jacobian. Let $R_{\lambda_{ij}}$ be the row of $J$ corresponding to the egde $\lambda_{ij} \in D$ and let $R_s$ be the row corresponding to $s$. Then after performing the sequence of row operation $R_s \to R_s - \frac{\lambda_{ij}}{2s} R_{\lambda_{ij}}$ and $R_s -> 2 R_s$, the entries of row $R_s$ are given by
\[
J_{s, K_{i,j}} = 
\begin{cases}
2, & \mathrm{if}~ i = j,\\
-\lambda_{ij}, & \mathrm{if}~ (i, j) \in D,\\
0, &\mbox{ } \mathrm{otherwise}.
\end{cases}
\]
\end{lemma}
\begin{proof}
The three cases are all similar and straightforward so we will only show the second case since it is the most interesting. So suppose that $(i, j) \in D$. By Lemma \ref{J}, the entry $J_{s, K_{ij}}$ before this sequence of row operations is
\[
J_{s, K_{ij}} = -\lambda_{ij} + \sum_{l \in \ch(i) \cap \ch(j) }\lambda_{il}\lambda_{jl}.
\]
After applying this sequence of row operations means we must add the quantity 
\begin{align*}
\sum_{(k,l) \in D} -\frac{\lambda_{kl}}{-2s} J_{\lambda_{kl}, K_{ij}} &= -\frac{\lambda_{ij}}{-2s}(-s) + \sum_{l \in \ch(i) \cap \ch(j)} -\frac{\lambda_{il}}{-2s} s\lambda_{jl} + \sum_{l \in \ch(i) \cap \ch(j)} -\frac{\lambda_{jl}}{-2s} s\lambda_{il}\\
&= \frac{\lambda_{ij}}{2} - \sum_{l \in \ch(i) \cap \ch(j)}\lambda_{il}\lambda_{jl}
\end{align*}
to $J_{s, K_{ij}}$. So we see that after adding this quantity the entry $J_{s, K_{ij}}$ will be $J_{s, K_{ij}} =  \frac{-\lambda_{ij}}{2}$ which completes the proof since our last operation multiplies this row by 2 which will yield the desired entry. 
\end{proof}

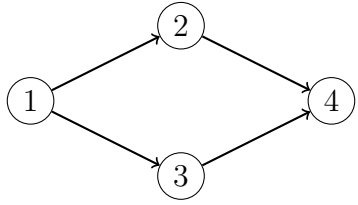
\begin{figure}
\centering
\begin{tikzpicture}
[main_node/.style={circle,fill=white,draw,minimum size=1.5 em,inner sep=2pt]}]

\node[main_node] (1) at (0, 0) {$1$};
\node[main_node] (2) at (2, 1) {$2$};
\node[main_node] (3) at (2, -1) {$3$};
\node[main_node] (4) at (4, 0) {$4$};

\draw [->, thick] (1) to (2);
\draw [->, thick] (1) to (3);
\draw [->, thick] (2) to (4);
\draw [->, thick] (3) to (4);

\end{tikzpicture}
\caption{The diamond graph $G$ which is used in Example \ref{ex_jac}.}
\label{fig:diamond}
\end{figure}

Every matroid is completely determined by its \emph{bases} which are the maximal independent sets with respect to inclusion. Each basis has the same cardinality which is called the \emph{rank} of the matroid and in this case it is simply the rank of the Jacobian evaluated at a generic point, or equivalently, the dimension of the model. 

\begin{theorem}\label{dim}
Let $G=(V,D)$ be a simple directed graph and $\mathcal{M}(\psi_G)$ be the Jacobian matroid of $\psi_G$. Then the Jacobian $J$ has full rank, thus any basis of the $\mathcal{M}(\psi_G)$ is of size $|D|+1$.
\end{theorem}

\begin{proof}
By Lemma \ref{J}, we know that each edge $(i,j)$ in $G$ leads to a $-s$ term in column $K_{ij}$. Now consider the $(|D|+1)\times(|D|+1)$ submatrix $J_S$ with columns corresponding to the set $S = \{K_{ij} ~|~ (i,j)\ {\rm or}\ (j,i) \in D\} \cup \{K_{ii}\}$ for any choice of $i \in V$. Then plugging in $\lambda_{ij} = 0$ and $s = 1$ in $J_S$ yields a constant multiple of the identity matrix up to row permutations, hence it has rank $|D|+1$. Since the rank of the Jacobian can only decrease when plugging in values for the parameters, it holds that $J_S$ has full rank over the fraction field $k(\lambda, s)$.  
\end{proof}

\begin{remark}
This theorem shows that the model of a simple directed graph is of expected dimension. However, a similar result does not hold for non-simple directed graphs, even if the number of edges is smaller than that of the complete graph.
\end{remark}

Since our main goal in this paper is to prove identifiability results, we are mainly interested in when two graphs $G_1$ and $G_2$ have different or the same Jacobian matroids. We end this section with several lemmas which give necessary conditions for when two graphs may yield the same matroid. 

\begin{lemma}\label{adj}
Let $G_1=(V,D_1)$, $G_2=(V,D_2)$ be two directed graphs with the same Jacobian matroid $\mathcal{M}$. If two node $i$ and $j$ are adjacent or have common children in one graph, then they must be adjacent or have common children in the other graph.
\end{lemma}
\begin{proof}
The column $K_{ij}$ is an independent set in the matroid $\cm$, if there exists a nonzero entry in this column which by Lemma \ref{J} happens if and only if $i$ and $j$ are adjacent or have common children. Since the two graphs have the same matroid, the result immediately follows. 
\end{proof}

\begin{lemma}\label{sinks_and_inedges}
Let $G_1=(V,D_1)$, $G_2=(V,D_2)$ be two directed graphs with the same Jacobian matroid $\mathcal{M}$. If the node $i$ is a sink node in both graphs, then $\pa_{G_1}(i) = \pa_{G_2}(i)$. 
\end{lemma}
\begin{proof}
Suppose that $k \in \pa_{G_1}(i)$, then by Lemma \ref{adj}, it must be that $k$ and $i$ are also adjacent or have common children in $G_2$ as well since they are adjacent in $G_1$ but $i$ is a sink node in $G_2$ so it cannot have any children. Thus it must be that $k \to i \in D_2$. 
\end{proof}

\section{Graphical Conditions for Distinguishing SEMs with Matroids}
\label{sec:DistinguishingWithMatroids}
In this section we provide several sufficient conditions for determining when two graphs have different matroids. We first show that two non-complete graphs with different out-degree sequences or two graphs with no unshielded colliders must have different matroids. We then provide a generalization of these two theorems which uses parentally closed sets to capture these conditions.

\begin{lemma}\label{rankupbd}
Let $G=(V,D)$ be a directed graph whose corresponding model has dimension $|D|+1$ and $J$ be the associated Jacobian. If $G$ is not complete, then for every node $i$ and any subset of the columns $S$ of size $|D|+1$ such that 
$S \cap \{K_{i1},K_{i2},...,K_{i(i-1)},K_{ii},K_{i(i+1)},...\}=\emptyset$, it holds that
$\rank(J_S) \leq |D| - |\ch(i)| + 1$.
\end{lemma}
\begin{proof}
Observe that for any choice of $k, l \neq i$, Lemma \ref{J} implies that $J_{\lambda_{ij}, K_{kl}} = 0$. Since every column $K_{kl} \in S$ satisfies this condition, it holds that all of the rows of $J_S$ which correspond to the edges $\lambda_{ij}$ are zero. This means there are at least $|\ch(i)|$ zero rows in $J_S$ and hence $\rank(J_S) \leq |D| - |\ch(i)| +1$. 
\end{proof}

\begin{lemma}\label{ranklowerbd}
Let $G=(V,D)$ be a simple directed graph and $J$ be the associated Jacobian. If $G$ is not complete, then for every node $i$, there exists a set of columns $S$ of size $|D|+1$ such that 
$S \cap \{K_{i1},K_{i2},\ldots,K_{i(i-1)},K_{ii},K_{i(i+1)},\ldots \}=\emptyset$ and $\rank(J_S) \geq |D|-|\ch(i)|+1$.
\end{lemma}
\begin{proof}
We start by constructing a potential set $S$. First we define the set
\[
S_E = \{K_{kl} ~|~ (k, l) \in D ~\mathrm{or}~ (l, k) \in D\}.
\]
and let $S_- \subseteq S_E$ be the subset consisting of columns of the form $K_{ki}$. So $S_-$ corresponds to the adjacent pairs $\{i,k\}$ and $|S_-|= \text{deg}(i)$ which will be removed from $S_E$ to obtain the required set $S$. We now split into two cases based on $\deg(i)$.

First consider the case where $\text{deg}(i)<|V|-1$. 
There is a node $j_0$ not adjacent to $i$. Let $S_+$ be a subset of size $(\text{deg}(i)+1)\leq |V|-1$ from $\{K_{11},K_{22},\ldots, K_{i-1,i-1},K_{i+1,i+1},\ldots\}$ containing all $K_{jj}$ such that $(j,i)\in D$ or $(i,j)\in D$ and $K_{j_0 j_0}$. The set $S=(S_E\backslash S_-)\cup S_+$ is of size $|D|+1$. Then the submatrix $J_S$ evaluated at $s=1, \lambda_{ji}=\varepsilon>0$ and all other edge weights $0$ is
\begin{align}\label{ranklb_mat}
&\begin{blockarray}{ccccccccc}
K_{j_{0}j_{0}} &\cdots & K_{j_{1}j_{1}} & \cdots & K_{j_{q}j_{q}} & K_{k_{1}l_{1}} & \cdots & K_{k_m l_m}\\
\begin{block}{(cccccccc)c}
  \times & \cdots & \times & \cdots & \times & -s &\cdots & \times & \lambda_{k_1 l_1} \\
  \vdots & \vdots & \vdots & \vdots & \vdots & \vdots &\ddots &\vdots & \vdots \\
  \times & \cdots & \times & \cdots & \times &\times & \cdots &-s & \lambda_{k_m l_m} \\
  \times & \cdots & 2s\lambda_{j_{1}i} & \cdots & 0 & \times & \cdots &\times & \lambda_{j_1 i}\\
  \vdots & \vdots & \vdots & \ddots & \vdots & \vdots & \vdots & \vdots & \vdots \\
  \times & \cdots & 0 & \cdots & 2s\lambda_{j_{q}i} & \times &\cdots &\times & \lambda_{j_q i} \\
  \times & \cdots & \times & \cdots & \times &\times & \cdots & \times & \lambda_{i n_1}\\
  \vdots & \vdots & \vdots & \vdots & \vdots &\vdots & \vdots & \vdots & \vdots\\
  \times & \cdots & \times & \cdots & \times &\times & \cdots & \times  & \lambda_{i n_p}\\
  1+\sum_{\alpha_0}\lambda_{j_{0}\alpha_{0}}^2 & \cdots & 1+\sum_{\alpha_1}\lambda_{j_{1}\alpha_{1}}^2 & \cdots & 1+\sum_{\alpha_q}\lambda_{j_{q}\alpha_{q}}^2 & \times &\cdots & \times & s\\
\end{block}
\end{blockarray}
\nonumber\\
=\quad &\begin{blockarray}{ccccccccc}
K_{j_{0}j_{0}} &\cdots & K_{j_{1}j_{1}} & \cdots & K_{j_{q}j_{q}} & K_{k_{1}l_{1}} & \cdots & K_{k_{m}l_{m}}\\
\begin{block}{(cccccccc)c}
  0 & \cdots & 0 & \cdots & 0 & -1 &\cdots & O(\varepsilon) & \lambda_{k_1 l_1}\\
  \vdots & \vdots & \vdots & \vdots & \vdots & \vdots &\ddots &\vdots & \vdots \\
  0 & \cdots & 0 & \cdots & 0 & O(\varepsilon) & \cdots &-1 & \lambda_{k_m l_n}\\
  0 & \cdots & 2\varepsilon & \cdots & 0 & \times & \cdots &\times & \lambda_{j_1 i}\\
  \vdots & \vdots & \vdots & \ddots & \vdots & \vdots & \vdots & \vdots & \vdots\\
  0 & \cdots & 0 & \cdots & 2\varepsilon & \times &\cdots &\times & \lambda_{j_q i}\\
  0 & \cdots & \times & \cdots & \times &\times & \cdots & \times & \lambda_{i n_1}\\
  \vdots & \vdots & \vdots & \vdots & \vdots &\vdots & \vdots & \vdots & \vdots\\
  0 & \cdots & \times & \cdots & \times &\times & \cdots & \times & \lambda_{i n_p} \\
  1 & \cdots & 2 & \cdots & 2 & 0 & \cdots & 0 & s\\
\end{block}
\end{blockarray}
\end{align}
Observe that after a reshuffling of the rows, the submatrix of $J_S$ corresponding to the rows labelled by 
$\{\lambda_{k_1 l_1}, \ldots, \lambda_{k_m l_m}, \lambda_{j_1 i}, \ldots, \lambda_{j_q i}, s\}$ is a block upper triangular matrix. The diagonal blocks correspond to the rows whose edges labels do not involve $i$ (those labelled by the $\lambda_{k_\alpha, l_\alpha}$), the parents of $i$, and $s$. Note that the block corresponding to the edges which do not involve $i$ is strictly diagonally dominant for a small enough choice of $\varepsilon$ and thus is full rank while the other two blocks are scalar multiples of the identity matrix. Thus we see that this submatrix of $J_S$ is full rank and so we have that 
$\rank(J_S) \geq (|D| -\deg(i)) + |\pa(i)| + 1 = |D| - |\ch(i)| + 1$.

Now suppose that $\text{deg}(i)=|V|-1$ and $i$ is not a sink node so there exists some $j_0 \in \ch(i)$. Let $S_+ = \{K_{11},K_{22},...,K_{(i-1)(i-1)},K_{(i+1)(i+1)},...\}$ and consider the set $S = (S_E \setminus S_-) \cup S_+$. Note that $S$ only contains $|D|$ columns now. However the submatrix $J_S$ has the exact same form as that pictured in Equation \ref{ranklb_mat} thus we again get that $\rank(J_S) \geq |D| - |\ch(i)|+1$.  Adding any additional column $K_{x y}$ such that $x,y \neq i$ to $S$ gives a set of the desired form and doing so can only increase the rank of $J_S$ so this case is complete. 

Lastly suppose that $\deg(i)=|V|-1$ and $i$ is a sink node in $G$. Once again take $S = (S_E \setminus S_-)\cup S_+$ where $S_+ = \{K_{11},K_{22},...,K_{(i-1)(i-1)},K_{(i+1)(i+1)},...\} \cup \{K_{x y}\}$ such that $x, y \neq i$ and $x, y$ are not adjacent in $G$. Since $G$ is not complete, such a pair of vertices is guaranteed to exist in $G$. 
Then evaluating the submatrix $J_S$  at $s=1, \lambda_{ji}=\varepsilon>0$ and all other edge weights $0$ yields

\begin{align*}
&
\begin{scriptsize}
\begin{blockarray}{ccccccccccc}
K_{xy} &\cdots & K_{xx} & K_{yy} & K_{j_{1}j_{1}} & \cdots & K_{j_{q}j_{q}} & K_{k_{1}l_{1}} & \cdots & K_{k_{m}l_{m}}\\
\begin{block}{(cccccccccc)c}
  \times & \cdots & \times & \times & \times & \cdots & \times & -s &\cdots & \times & \lambda_{k_1 l_1}\\
  \vdots & \vdots & \vdots & \vdots & \vdots & \vdots & \vdots & \vdots &\ddots &\vdots & \vdots\\
  \times & \cdots & \times & \times & \times & \cdots & \times &\times & \cdots &-s & \lambda_{k_m l_m} \\
  s \lambda_{ni} & \cdots & 2s\lambda_{mi} & 0 & 0 & \cdots & 0 & \times & \cdots &\times & \lambda_{x i}\\
  s\lambda_{mi} & \cdots & 0 & 2s\lambda_{ni} & 0 & \cdots & 0 & \times & \cdots &\times & \lambda_{y i} \\
  \times & \cdots & 0 & 0 & 2s\lambda_{j_{1}i} & \cdots & 0 & \times & \cdots &\times & \lambda_{j_1 i}\\
  \vdots & \vdots & \vdots & \vdots & \vdots & \ddots & \vdots & \vdots & \vdots & \vdots & \vdots\\
  \times & \cdots & 0 & 0 & 0 & \cdots & 2s\lambda_{j_{q}i} & \times &\cdots &\times & \lambda_{j_q i}\\
  \times & \cdots & \times & \times & \times & \cdots & \times &\times & \cdots & \times & \lambda_{i n_1}\\
  \vdots & \vdots & \vdots & \vdots & \vdots & \vdots & \vdots &\vdots & \vdots & \vdots & \vdots \\
  \times & \cdots & \times & \times & \times & \cdots & \times &\times & \cdots & \times & \lambda_{i n_p}\\
  \sum_{\alpha_x,\alpha_y}\lambda_{x \alpha_{x}}\lambda_{y \alpha_y} & \cdots & 1+\sum_{\alpha_x}\lambda_{x \alpha_{x}}^2 & 1+\sum_{\alpha_y}\lambda_{y \alpha_{y}}^2 & 1+\sum_{\alpha_1}\lambda_{j_{1}\alpha_{1}}^2 & \cdots & 1+\sum_{\alpha_q}\lambda_{j_{q}\alpha_{q}}^2 & \times &\cdots & \times & s\\
\end{block}
\end{blockarray}
\end{scriptsize}
\\
=\quad
&\begin{blockarray}{ccccccccccc}
K_{xy} &\cdots & K_{xx} & K_{yy} & K_{j_{1}j_{1}} & \cdots & K_{j_{q}j_{q}} & K_{k_{1}l_{1}} & \cdots & K_{k_{m}l_{m}}\\
\begin{block}{(cccccccccc)c}
  0 & \cdots & 0 & 0 & 0 & \cdots & 0 & -1 &\cdots & O(\varepsilon) & \lambda_{k_1 l_1} \\
  \vdots & \vdots & \vdots & \vdots & \vdots & \vdots & \vdots & \vdots &\ddots &\vdots  & \vdots\\
  0 & \cdots & 0 & 0 & 0 & \cdots & 0 & O(\varepsilon) & \cdots &-1 & \lambda_{k_m l_m} \\
  1 & \cdots & 2\varepsilon & 0 & 0 & \cdots & 0 & \times & \cdots &\times & \lambda_{x i}\\
  1 & \cdots & 0 & 2\varepsilon & 0 & \cdots & 0 & \times & \cdots &\times & \lambda_{y i}\\
  0 & \cdots & 0 & 0 & 2\varepsilon & \cdots & 0 & \times & \cdots &\times  & \lambda_{j_1 i}\\
  \vdots & \vdots & \vdots & \vdots & \vdots & \ddots & \vdots & \vdots & \vdots & \vdots & \vdots\\
  0 & \cdots & 0 & 0 & 0 & \cdots & 2\varepsilon & \times &\cdots &\times & \lambda_{j_q i} \\
  0 & \cdots & \times & \times & \times & \cdots & \times &\times & \cdots & \times &  \lambda_{i n_1}\\
  \vdots & \vdots & \vdots & \vdots & \vdots & \vdots & \vdots &\vdots & \vdots & \vdots & \vdots \\
  0 & \cdots & \times & \times & \times & \cdots & \times &\times & \cdots & \times  &  \lambda_{i n_p} \\
  1 & \cdots & 2 & 2 & 2 & \cdots & 2 & 0 & \cdots & 0 & s\\
\end{block}
\end{blockarray}
\end{align*}
The two entries of column $K_{xy}$ in rows $\lambda_{xi}$ and $\lambda_{yi}$ can be eliminated by subtracting a multiple of columns $\{K_{xx},K_{yy}\}$. This may introduce nonzero entries in column $K_{xy}$ in the rows corresponding to the children of $i$ but these can then be eliminated using the row corresponding to $s$. After these operations then $J_S$ will again have the same form as that shown in Equation \ref{ranklb_mat} and so this case is also complete. 
\end{proof}

\begin{remark}
Lemma \ref{ranklowerbd} also holds for non-simple graphs which satisfy $\text{deg}(i)< |V|$, $|D|\leq \binom{|V|}{2}$, and have expected dimension $|D|+1$. 
\end{remark}

The following example illustrates the proofs of Lemma \ref{rankupbd} and Lemma \ref{ranklowerbd}.

\begin{example}
We again consider the graph pictured in Example \ref{ex:linearCyclicSEM} and set $i=3$ which satisfies $\deg(3) = 2 < 3 = |V| - 1$. 
The Jacobian $J(\psi)$ is
\[
\begin{blockarray}{ccccccccccc}
K_{11} & K_{22} & K_{33} & K_{44} & K_{12} & K_{23} & K_{34} & K_{13} & K_{24} & K_{14} \\
\begin{block}{(cccccccccc)c}
  2s\lambda_{12} & 0 & 0 & 0    & -s & 0 & 0 & 0 & 0 & 0 & \lambda_{12}\\
  0 & 2s\lambda_{23} & 0 & 0    & 0 & -s & 0 & 0 & 0 & 0 & \lambda_{23}\\
  0 & 2s\lambda_{24} & 0 & 0    & 0 & s\lambda_{34} & 0 & 0 & -s & 0 & \lambda_{24}\\
  0 & 0 & 2s\lambda_{34} & 0    & 0 & s\lambda_{24} & -s & 0 & 0 & 0 & \lambda_{34}\\
  1+\lambda_{12}^2 & 1+\lambda_{23}^2+\lambda_{24}^2 & 1+\lambda_{34}^2 & 1 & -\lambda_{12} & -\lambda_{23}+\lambda_{24}\lambda_{34} & -\lambda_{34} & 0 & -\lambda_{24} & 0 & s\\
\end{block}
\end{blockarray}
\]
Lemma \ref{ranklowerbd} guarantees that there exists a set $S \cap \{K_{13}, K_{23}, K_{33}, K_{34}\} = \emptyset$ and $\rank(J_S) \geq |D| - |\ch(3)| + 1 = 4 - 1 + 1 = 4$. 
From the proof of the lemma, the set $S$ which is constructed is $S = (S_E \setminus S_-) \cup S_+$ where
\begin{align*}
S_E &= \{K_{12}, K_{23}, K_{34}, K_{24}\},\\
S_- &= \{K_{23}, K_{34} \}, \\
\end{align*}
and $S_+$ can be any subset of size $\deg(3)+1=3$ from $\{K_{11}, K_{22}, K_{44}\}$ which means it must be the entire set. Thus $S = \{K_{11}, K_{22}, K_{44}, K_{12}, K_{24}\}$ and the submatrix $J_S$ is
\[
\begin{blockarray}{cccccc}
K_{11} & K_{22} & K_{44} & K_{12} & K_{24} \\
\begin{block}{(ccccc)c}
2s\lambda_{12} & 0 & 0    & -s & 0 & \lambda_{12}\\
0 & 2s\lambda_{24} & 0    & 0 & -s & \lambda_{24}\\
0 & 2s\lambda_{23} & 0    & 0 & 0 & \lambda_{23}\\
0 & 0 & 0    & 0 & 0 & \lambda_{34}\\
1+\lambda_{12}^2 & 1+\lambda_{23}^2+\lambda_{24}^2 & 1 & -\lambda_{12} & -\lambda_{24} & s\\
\end{block}
\end{blockarray}
\]
Observe that if we substitute in $s = 1$, $\lambda_{ji} = \epsilon$, and let all other edge weights be zero, then the submatrix becomes
\[
\begin{blockarray}{cccccc}
K_{11} & K_{22} & K_{44} & K_{12} & K_{24} \\
\begin{block}{(ccccc)c}
0 & 0 & 0    & -1 & 0 & \lambda_{12}\\
0 & 0 & 0    & 0 & -1 & \lambda_{24}\\
0 & 2\varepsilon & 0 & 0 & 0 & \lambda_{23}\\
0 & 0 & 0    & 0 & 0 & \lambda_{34}\\
1 & 1+\varepsilon & 1 & 0 & 0 & s\\
\end{block}
\end{blockarray}
\]
and we can now see that columns $K_{11}, K_{22}, K_{12}, K_{44}$ are linearly independent thus $\rank(J_S) \geq 4$. On the other hand observe that the row corresponding to $\lambda_{34}$ is a zero row and so $\rank(J_S)\leq 4$ which is guaranteed by Lemma \ref{rankupbd}. 
\end{example}

We are now ready to state our first major theorem which is mainly powered by the previous two lemmas. 

\begin{theorem}\label{outd}
\textup{\textbf{(Outdegree theorem)}}
Let $G_1=(V,D_1)$, $G_2=(V,D_2)$ be two simple directed graphs. If one of the graphs is not complete and there exists a node $i \in V$ such that $|\ch_{1}(i)| \neq |\ch_{2}(i)|$ then $G_1$ and $G_2$ have different Jacobian matroids.
\end{theorem}

\begin{proof}
If $|D_1|\neq |D_2|$, the two models must have different dimension and hence different matroids so we assume that $|D_1| = |D_2|$. 
Without loss of generality assume that $|\ch_1(i)|>|\ch_2(i)|$. By Lemma \ref{ranklowerbd}, there exists a column set $S=(S_E\backslash S_-)\cup S_+$ such that $\rank(J^{(2)}_S)\geq |D_2|-|\ch_2(i)|+1$ but by Lemma \ref{rankupbd} we have that $\rank(J^{(1)}_S)\leq |D_1|-|\ch_1(i)|+1$ which yields
\[
\rank(J^{(2)}_S) \geq |D_2| - |\ch_2(i)|  + 1 > |D_1| - |\ch_1(i)| + 1 \leq \rank(J^{(1)}_S). 
\]
Thus it must be that $J^{(1)}$ and $J^{(2)}$ have different matroids. 
\end{proof}

Since this theorem is mainly powered by Lemma \ref{ranklowerbd}, it can also be extended to non-simple graphs if they have expected dimension $|D|+1$ and the degree of each node is at most $|V|-1$ but we will omit the details since we are primarily focused on simple graphs. The previous theorem immediately implies some identifiability results which we will summarize at the end of this section. 
Essentially though, it tells us that any two graphs with different out-degree sequences are distinguishable. 
The theorem also gives us the following corollary for graphs which do yield the same Jacobian matroid.

\begin{corollary}
If two simple directed graphs have the same Jacobian matroid, and at least one of them is not complete, then they must have the same sink nodes.
\end{corollary}

While Theorem \ref{outd} is able to distinguish a very large proportion of graphs, there may be pairs of graphs with the same out-degree sequences but different matroids. Our next theorem is based on the following definition. 

\begin{definition}
Let $G = (V, D)$ be a directed graph. A \emph{transitive triangle} or \emph{shielded collider} in $G$ is a triple of nodes $i,j,k \in V$ such that the induced subgraph of $G$ on $i,j,k$ is of one of the forms pictured in Figure \ref{fig:transitiveTriangles}. Alternatively, it means that for all $j\in V$ and for all $i\in \ch(j)$ it holds that $\ch(j)\cap \ch(i)=\emptyset$. If a graph has no transitive triangles then we say it is \emph{transitive triangle free}. 
\end{definition}

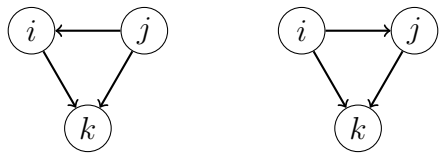
\begin{figure}
    \centering
\begin{tikzpicture}
[main_node/.style={circle,fill=white,draw,minimum size=1.5 em,inner sep=0pt]}]

\node[main_node] (0) at (0, 0) {$i$};
\node[main_node] (1) at (1.5, 0) {$j$};
\node[main_node] (2) at (0.75, -1.5*1.732/2) {$k$};

\draw[->, thick] (0) to (2);
\draw[->, thick] (1) to (2);
\draw[->, thick] (1) to (0);

\node at (3,0) {};
\end{tikzpicture}
\begin{tikzpicture}
[main_node/.style={circle,fill=white,draw,minimum size=1.5 em,inner sep=0pt]}]

\node[main_node] (0) at (0, 0) {$i$};
\node[main_node] (1) at (1.5, 0) {$j$};
\node[main_node] (2) at (0.75, -1.5*1.732/2) {$k$};

\draw[->, thick] (0) to (2);
\draw[->, thick] (1) to (2);
\draw[->, thick] (0) to (1);
\end{tikzpicture}
    \caption{The two forbidden subgraph structures for transitive triangle (shielded collider) free graphs.}
    \label{fig:transitiveTriangles}
\end{figure}

\begin{theorem}\label{notri}
Let $G_1=(V,D_1)$, $G_2=(V,D_2)$ be different, transitive triangle-free, non-complete, simple directed graphs with node set $V$. Then $G_1$ and $G_2$ have different Jacobian matroids.
\end{theorem}
\begin{proof}
We again proceed by constructing a subset $S$ such that the induced submatrices $J^{(1)}_S$ and $J^{(2)}_S$ have different ranks. If $G_1$ and $G_2$ satisfy the condition in Theorem \ref{outd}, the conclusion obviously holds. Otherwise every node has the same out-degree in two graphs. Since $G_1$ and $G_2$ are different, there must exist a node $i$ such that $\ch_1(i)\neq \ch_2(i)$. So assume that $(i,j)\in D_1$ but $(i,j)\notin D_2$. We may also assume that either $(j,i)\in D_2$ or $\ch_2(i)\cap \ch_2(j)\neq\emptyset$ or else Lemma \ref{adj} would immediately imply that the matroids are different.

In the case that $(j,i)\notin D_2$, then $\deg_2(j)<|V|-1$ and there exists some node $l\in Ch_2(i)\cap Ch_2(j)$. Let $S = (S_E \setminus S_-) \cup S_+$ where
\begin{align*}
&S_E = \{K_{km} ~|~ (k, m) \in D_2 ~\mathrm{or}~ (m, k) \in D_2\}, \\
&S_- = \{K_{kj} ~|~ (k, j) \in D_2 ~\mathrm{or}~ (j, k) \in D_2\}, \\
\end{align*}
and $S_+$ be a subset of size $(\text{deg}(j)+1)\leq |V|-1$ from $\{K_{11},K_{22},\ldots, K_{j-1,j-1},K_{j+1,j+1},\ldots\}$ containing all $K_{m m}$ such that $(m,j)\in D_2$ or $(j, m)\in D_2$ and $K_{ii}$. Observe that this exactly the set $S$ which is constructed in Lemma \ref{ranklowerbd} (except $i$ has been replaced with $j$ and $j_0$ with $i$). Thus by Lemmas \ref{rankupbd} and \ref{ranklowerbd} it holds that
$$
\rank(J^{(2)}_S) \geq |D_2|-|\ch_2(j)|+1=|D_1|-|\ch_1(j)|+1 \geq \rank(J^{(1)}_S).
$$
Note that $K_{ll}\in S$ but $K_{ij}\notin S$ so we replace $K_{ll}$ by $K_{ij}$ to obtain a new set $S'$ which will yield a higher rank submatrix of $J^{(2)}$ without increasing the rank of $J^{(1)}$. To see this consider $J^{(2)}$ evaluated at $s=1,\lambda_{il}=1, \lambda_{mj}=1$ where $m \in \pa(j)$ and all other edge weights $0$ which yields the block matrix
\begin{align}\label{ttf_J2}
J^{(2)}_S=
\begin{blockarray}{cccccc}
K_{ii} & K_{m_1 m_1}, \ldots, K_{m_p m_p} & K_{ij} & K_{j_1 j_1 }, \ldots, K_{j_q j_q} & K_{k_1 l_1} \ldots K_{k_n l_n}\\
\begin{block}{(ccccc)c}
\textbf{a} & \mathbf{0} & 0 & \mathbf{0} & -I_{|D|-\text{deg}(j)} & \lambda_{k_\alpha l_\alpha},~ k_\alpha, l_\alpha \neq j\\
\textbf{0} & \mathbf{0} & 0 &  2I_{|\pa(j)|} & \mathbf{0} & \lambda_{j_\alpha,j},~ j_\alpha \in \pa(j)\\
0 & \textbf{0}^T & 1 & \textbf{0}^T & \textbf{0}^T & \lambda_{jl}\\
\vdots & \vdots & \vdots & \vdots & \vdots & \vdots\\
1 & \times & \times & \times & \textbf{0}^T & s\\
\end{block}
\end{blockarray}.
\end{align}
The rightmost pivot block was previously of the form $I+\varepsilon A$, where the second terms comes from common children between $k_\alpha$ and $l_\alpha$. The transitive triangle-free assumption now makes the second term zero. Now observe that the vector $\textbf{a}$ in the upper left hand corner can be eliminated using the columns coming from $K_{k_\alpha l_\alpha}$ block which implies that
$\rank(J^{(2)}_{S'}) \geq |D_2| - \deg(j) + |\pa(j)|  + 2 = |D_2| - |\ch(j)| + 2$. 
On the other hand, if $x \in \ch_1(j)$ then the entry $(J^{(1)}_{S'})_{\lambda_{jx, K_{ij}}} =0$ since $i \notin \ch_1(j)$ and $\ch_1(i) \cap \ch_1(j) = \emptyset$ since $G_1$ is transitive triangle-free. Hence it still holds that $\rank(J^{(1)}_{S'}) \leq |D_1| - |\ch(i)| + 1$ which completes this case. 

If $(j,i)\in D_2$, then the same construction of $S$ which is used in Theorem \ref{outd} can be used here since it must be that either  $\deg_2(j) <|V|-1$ or $\deg_2(j)=|V|-1$ but $j$ is not a sink node. In either case, one can take $S$ and then replace $K_{ii}$ with $K_{ij}$ to obtain $S'$. The argument then proceeds identically to the previous case so we omit the details. 
\end{proof}

The following example illustrates the previous theorem. In particular, it shows exactly how one may find a set $S$ which verifies that the matroids of two graphs with no transitive triangles differ. 

\begin{example}\label{eg_ttf}
Consider the graphs $G_1$ and $G_2$ in Figure \ref{fig:sameOutDegSeq} which are the same except for the change in the outgoing edges from node $1$. These two graphs clearly have the same out-degree sequence which is $(3, 1, 1, 0, 1, 1)$. This means that Theorem \ref{outd} cannot be applied to certify that the $G_1$ and $G_2$ have different Jacobian matroids. We can instead use the construction in Theorem \ref{notri} to find a suitable set $S$ which will distinguish the matroids.

In this case we have the $\ch_1(1) \neq \ch_2(1)$ so $i = 1$. Furthermore, the edge $(1, 2) \in D_1$ but $(1,2) \notin D_2$ so $j = 2$. This means that $S = (S_E \setminus S_-) \cup S_+$ where
\begin{align*}
&S_E = \{K_{km} ~|~ (k, m) \in D_2 ~\mathrm{or}~ (m, k) \in D_2\} 
= \{K_{14}, K_{15}, K_{16}, K_{25}, K_{26}, K_{35}, K_{36}\}, \\
&S_- = \{K_{kj} ~|~ (k, j) \in D_2 ~\mathrm{or}~ (j, k) \in D_2\} 
= \{K_{25}, K_{26}\},\\
&S_+ = \{K_{55}, K_{66}\}. 
\end{align*}
Note that in the language of Theorem \ref{notri}, the common child of $i=1$ and $j=2$ is $l = 5$.   Now according to Theorem \ref{notri}, the set which will actually yield a rank difference is the set $S'$ which is obtained by replacing $K_{ll} = K_{55}$ with $K_{ij} = K_{12}$. This yields the set
\[
S' = \{K_{11}, K_{12}, K_{66}, K_{14}, K_{15}, K_{16}, K_{35}, K_{36}\}
\]
which indeed has $|D_2|+1 = 8$ elements as is desired. Then $J_S^{(2)}$ has the form
\begin{equation*}
J_S^{(2)}=
\begin{blockarray}{ccccccccc}
K_{11} & K_{12} & K_{66} & K_{14} & K_{15} & K_{16} & K_{53} & K_{36} \\
\begin{block}{(cccccccc)c}
2s\lambda_{14}& 0& 0& -s& 0& 0& 0& 0 & \lambda_{14} \\
2s\lambda_{15}& s\lambda_{25}& 0& 0 & -s& 0& 0& 0& \lambda_{15} \\ 
2s\lambda_{16}& 0& 0& 0& 0& -s& 0& 0 & \lambda_{16} \\
0& 0& 0& 0& 0& 0& 0& -s & \lambda_{36} \\
0& 0& 0& 0& 0& 0& -s& 0 & \lambda_{53} \\
0& 0& 2s\lambda_{62}& 0& 0& 0& 0& 0 & \lambda_{62} \\
0& s\lambda_{15}& 0& 0& 0& 0& 0& 0& \lambda_{25} \\
2& 0& 2& -\lambda_{14}& -\lambda_{15}& -\lambda_{16}& -\lambda_{53}& -\lambda_{36} & s\\
\end{block}
\end{blockarray}.
\end{equation*}
It is clear that this matrix has rank $8$ while $J_S^{(1)}$ only has rank 7. 
\end{example}

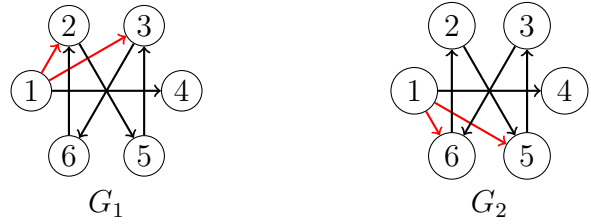
\begin{figure}
\centering
\begin{minipage}[t]{0.3\textwidth}
\centering
\begin{tikzpicture}
[main_node/.style={circle,fill=white,draw,minimum size=1 em,inner sep=2pt]}]

\node[main_node] (1) at (0, 0) {$1$};
\node[main_node] (2) at (0.5, 0.866) {$2$};
\node[main_node] (3) at (1.5, 0.866) {$3$};
\node[main_node] (4) at (2, 0) {$4$};
\node[main_node] (5) at (1.5, -0.866) {$5$};
\node[main_node] (6) at (0.5, -0.866) {$6$};

\draw [->, thick, red] (1) to (2);
\draw [->, thick, red] (1) to (3);
\draw [->, thick] (1) to (4);
\draw [->, thick] (2) to (5);
\draw [->, thick] (3) to (6);
\draw [->, thick] (5) to (3);
\draw [->, thick] (6) to (2);

\node (7) at (1, -1.5) {$G_1$};
\end{tikzpicture}
\end{minipage}
\begin{minipage}[t]{0.3\textwidth}
\centering
\begin{tikzpicture}
[main_node/.style={circle,fill=white,draw,minimum size=1.5 em,inner sep=2pt]}]

\node[main_node] (1) at (0, 0) {$1$};
\node[main_node] (2) at (0.5, 0.866) {$2$};
\node[main_node] (3) at (1.5, 0.866) {$3$};
\node[main_node] (4) at (2, 0) {$4$};
\node[main_node] (5) at (1.5, -0.866) {$5$};
\node[main_node] (6) at (0.5, -0.866) {$6$};

\draw [->, thick, red] (1) to (5);
\draw [->, thick, red] (1) to (6);
\draw [->, thick] (1) to (4);
\draw [->, thick] (2) to (5);
\draw [->, thick] (3) to (6);
\draw [->, thick] (5) to (3);
\draw [->, thick] (6) to (2);

\node (7) at (1, -1.5) {$G_2$};
\end{tikzpicture}
\end{minipage}
\caption{Two graphs with the same out-degree sequence but which have no transitive triangles (shielded colliders). }
\label{fig:sameOutDegSeq}
\end{figure}

Our last major theorem in this section provides a much stronger generalization of Theorem \ref{outd} and Theorem \ref{notri}. While this condition will allow us to distinguish many more pairs of graphs than Theorems \ref{outd} or \ref{notri}, it can be much more difficult tot check. Before introducing this theorem we need the following definition. 

\begin{definition}
Let $G = (V, D)$ be a directed graph and $i \in V$. A subset $L \subseteq \neighbors(i)$ is \emph{parentally closed} with respect to $i$ if $\pa(L) \cap \neighbors(i) \subseteq L$.
We denote the collection of all parentally closed sets with respect to $i$ by $\mathcal{L}_i$. 
\end{definition}

\begin{theorem}\label{outd_layer}
Let $G_1=(V,D_1)$, $G_2=(V,D_2)$ be simple directed graphs and which are not complete. Let $i \in V$ and $\mathcal{L}_i^k$ be the collection of parentally closed sets with respect to $i$ for graph $G_k$. 
If there exists a set $L\in \mathcal{L}_i^k$ such that $|\ch_{k}(i)\cap L|>|\ch_{3-k}(i)\cap L|$, $k\in\{1,2\}$, then $G_1$ and $G_2$ have different Jacobian matroids.
\end{theorem}

\begin{proof}
The proof follows similarly to that of Theorems \ref{outd} and \ref{notri}. We will again construct a set $S$ for which the rank function of the corresponding matroids differ. 

Suppose that the inequality $|\ch_{1}(i)\cap L|>|\ch_{2}(i)\cap L|$ holds for some $L\in\mathcal{L}_i^1$. If $L=V\backslash\{i\}$, then we are in the case of Theorem \ref{outd}. Otherwise we can assume that $|L|<|V|-1$. Under this condition, we set 
\begin{align*}
& S_E=\{K_{mn} ~|~ (m,n)\ {\rm or}\ (n,m) \in D_2\}, \\
&S_-=\{K_{im} ~|~ m\in L\}, \\
&S_+=\{K_{mm} ~|~  (m,i)\ {\rm or}\ (i,m)\in D_2, m\in L\}\cup\{K_{j_0j_0}\}, 
\end{align*}
where $j_0$ is an arbitrary node in $V\backslash(\{i\}\cup L)$. The size of the column set $S=(S_E\backslash S_-)\cup S_+$ is $|D_2|-|L\cap \neighbors_2(i)| + |L\cap \neighbors_2(i)| +1 = |D_2|+1$, hence the submatrices $J_S^{(1)},J_S^{(2)}$ are of size $(|D_2|+1)\times (|D_2|+1)$. 

For each node $j\in \ch_1(i)\cap L$, nonzero entries of $J^{(1)}$ in the row $\lambda_{ij}$ may only appear in columns $K_{ii}, K_{ij}$ or $K_{il}$, for $l \in \pa_1(j)$.  We claim that none of these columns are contained in $S$. By construction we have that $K_{ii} \notin S, S_+$ and $K_{ij} \in S_-$ since the parental closure of $L$ implies that $j \in L$. If $K_{il} \in S_E$  then $l \in \neighbors(i)$ and thus $l \in \neighbors(i) \cap \pa_1(j) \subseteq L$ so $l \in L$ (this is pictured in Figure \ref{fig:parentallyClosedSetProof}). This implies that $K_{il} \in S_-$ and so it holds that $K_{il} \notin S$. As a result we get that there are $|\ch_1(i) \cap L|$ zero rows of $J_S^{(1)}$ so $\rank(J_S^{(1)}) \leq |D_1| - |\ch_1(i) \cap L| + 1$.

Next, we show that for certain values of the parameters, $\rank(J_S^2) \geq |D_2|+1-|\ch_2(i)\cap L|$. For each $l_q \in \pa_2(i) \cap L$, let $\lambda_{i l_q}= \varepsilon$, let $s=1$, and all other edge weights be zero. Then the submatrix $J_S^2$ takes similar form as that in \eqref{ranklb_mat}, but with $|\ch_2(i)\cap L|$ uncertain rows instead of $|\ch_2(i)|$. This implies that
\[
\rank(J_S^{(2)}) \geq |D_2| - |\ch_2(i) \cap L| + 1 > |D_1| - |\ch_1(i) \cap L| + 1 \geq \rank(J_S^{(1)})
\]
which completes the proof. 
\end{proof}

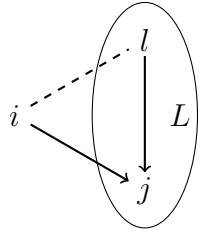
\begin{figure}
\centering
\begin{tikzpicture}
[main_node/.style={draw=none,fill=none, minimum size=1 em, inner sep=1.5pt}]

\node[main_node] (0) at (0, 0) {$i$};
\node[main_node] (1) at (1.732, -1) {$j$};
\node[main_node] (2) at (1.732, 1) {$l$};

\draw[->, thick] (0) to (1);
\draw[-, thick, dashed] (0) to (2);
\draw[->, thick] (2) to (1);
\draw[rotate=0] (1.732, 0) ellipse (0.7cm and 1.5cm) node at (2.2, 0) {$L$};
\end{tikzpicture}
\caption{This displays the subgraph relating $i, j$, and $l$ in the proof of Theorem \ref{outd_layer}. Since $L$ is parentally closed and $j \in L$ and $l \in \pa(j) \cap \neighbors(i)$, it must be that $l \in L$. }
\label{fig:parentallyClosedSetProof}
\end{figure}

Theorems \ref{outd} and \ref{notri} are actually special cases of Theorem \ref{outd_layer}. The former only checks the trivial parent closed set $L= \neighbors(i)$, while the latter considers the powersets $\mathcal{L}_1= \mathcal{P}(\ch_1(i))$ and $\mathcal{L}_2=\mathcal{P}(\ch_1(i))$. While it is computationally more complex to check this condition, it allows one to distinguish many more pairs of graphs than the previous two criteria. This is illustrated in the following example. 

\begin{example}
Let $G_1$ and $G_2$ be the graphs pictured in Figure \ref{fig:difParentallyClosedGraphs}. First observe that both $G_1$ and $G_2$ have the out-degree sequence $(2,2,0,0)$ and both graphs contain a transitive triangle so Theorems \ref{outd} and \ref{notri} cannot be applied to distinguish these graphs but Theorem \ref{outd_layer} does apply. 
The set $L = \{l_2, j\}$ is parentally closed in $G_2$ with respect to $i$ since $\pa(L) \cap \neighbors(i) = \{j\} \subseteq L$ and we can see that
\[
|\ch_2(i) \cap L| = 2 > 1 = |\ch_1(i) \cap L|
\]
thus by Theorem \ref{outd_layer}, these two models are distinguishable. 
\end{example}

\begin{figure}
\centering
\begin{minipage}[t]{0.3\textwidth}
\centering
\begin{tikzpicture}
[main_node/.style={draw=none,fill=none,minimum size=1 em, inner sep=1.5pt}]

\node[main_node] (0) at (0, 0) {$j$};
\node[main_node] (1) at (1.5, 0) {$i$};
\node[main_node] (2) at (0, -1.5) {$l_1$};
\node[main_node] (3) at (1.5, -1.5) {$l_2$};

\draw[->, thick] (0) to (1);
\draw[->, thick] (0) to (2);
\draw[->, thick] (1) to (2);
\draw[->, thick] (1) to (3);
\end{tikzpicture}
\caption*{$G_1$}
\end{minipage}
\begin{minipage}[t]{0.3\textwidth}
\centering
\begin{tikzpicture}
[main_node/.style={draw=none,fill=none,minimum size=1 em, inner sep=1.5pt}]

\node[main_node] (0) at (0, 0) {$j$};
\node[main_node] (1) at (1.5, 0) {$i$};
\node[main_node] (2) at (0, -1.5) {$l_1$};
\node[main_node] (3) at (1.5, -1.5) {$l_2$};

\draw[->, thick] (1) to (0);
\draw[->, thick] (0) to (2);
\draw[->, thick] (0) to (3);
\draw[->, thick] (1) to (3);
\end{tikzpicture}
\caption*{$G_2$}
\end{minipage}
\caption{Two graphs which are not distinguished by Theorems \ref{outd} and \ref{notri} but are distinguished by Theorem \ref{outd_layer}.}
\label{fig:difParentallyClosedGraphs}
\end{figure}

The next theorem is purely graph-theoretic, but together with Theorem \ref{outd_layer} can be very powerful for distinguishing graphs. 

\begin{theorem}
\label{thm:difGraphsPC_Set}
Let $G_1 = (V, D_1)$ and $G_2 = (V, D_2)$ be non-complete, simple directed graphs with the same out-degree sequence such that $D_1 \neq D_2$. If one of the graphs is acyclic, then there exists a node $i \in G_k$ and a set $L$ which is parentally closed with respect to $i$ such that $|\ch_{k}(i)\cap L|>|\ch_{3-k}(i)\cap L|$. 
\end{theorem}
\begin{proof}
Without loss of generality assume $G_1$ is a directed acyclic graph which is topologically ordered and let $<$ be a linear extension of the topological ordering. Let $i$ be the minimal node with respect to $<$ such that $\ch_1(i) \neq \ch_2(i)$ and let $j \in \ch_1(i) \setminus \ch_2(i)$. Note that such a node is guaranteed to exist since the two graphs have the same out-degree sequence, meaning that $|\ch_1(i)| = |\ch_2(i)|$. Then consider the set $L = \an_1(j) \cap \neighbors_1(i)$ which is parentally closed since if $k \in \pa_1(L) \cap \neighbors_1(i)$ then there exists some $k' \in L$ such that $k \in \pa(k')$ which immediately implies $k \in \an_1(j)$ and thus $k \in L$. 

We claim that $|\ch_{1}(i)\cap L|>|\ch_{2}(i)\cap L|$. We do this by showing that
\[
\ch_2(i) \cap L \subseteq \ch_1(i) \cap L
\]
which will immediately prove the claim since $j \in \ch_{1}(i)\cap L$ but cannot be in $\ch_{2}(i)\cap L$ by construction. So let $k \in \ch_2(i) \cap L = \ch_2(i) \cap \an_1(j) \cap \neighbors_1(i)$. This means $k \in \neighbors_1(i)$ so either $k \in \pa_1(i)$ or $k \in \ch_1(i)$. If $k \in ch_1(i)$ then we are done so suppose that $k \in \pa_1(i)$. Then $i \in \ch_1(k)$ but $i \notin \ch_2(k)$ since $k \in \ch_2(i)$ and $G_2$ is simple. This implies $\ch_1(k) \neq \ch_2(k)$ which is a contradiction since $k < i$ so by assumption $\ch_1(k) = \ch_2(k)$. Thus we see that $\ch_2(i) \cap L \subseteq \ch_1(i) \cap L$
and so the result follows. 
\end{proof}

\begin{corollary}
Let $G_1$ and $G_2$ be distinct digraphs with the same out-degree sequence such that at least one is acyclic. Then $G_1$ and $G_2$ have different Jacobian matroids. 
\end{corollary}

This corollary immediately follows from combining Theorems \ref{outd_layer} and \ref{thm:difGraphsPC_Set}. It also immediately provides an alternate proof for the identifiability of homoscedastic DAG models which was established in \cite{ideneqPeters,eqvar/biomet/asz049} while also showing that any cyclic graph is always distinguishable from a DAG. We end this section with the following corollary which summarizes the identifiability results given by the previous theorems. 

\begin{corollary}
 Let $\mathcal{G}$ be the collection of non-complete simple directed graphs. If the collection satisfies one of the following conditions, then the discrete parameter $G$ of the family $\{M_G\}_{G \in \mathcal{G}}$ of homoscedastic linear structural equation models is generically identifiable:
\begin{enumerate}[(i)]
    \item Every graph $G\in\mathcal{G}$ has a unique out-degree sequence;
    \item Every graph $G\in\mathcal{G}$ is transitive triangle-free;
    \item For every pair of graphs $G_1, G_2 \in\mathcal{G}$, there exists a node $i$ and a parentally closed set $L$ such that $|\ch_1(i) \cap L| \neq |\ch_2(i) \cap L|$. 
\end{enumerate}   
\end{corollary}

\section{Conjectures and Computational Study of Simple Graphs}
\label{sec:Conjectures}
In this section we outline a few conjectures concerning these models and provide some supporting computational evidence. In particular, we conjecture that Theorem \ref{outd_layer} can be used to distinguish many more non-complete graphs. Moreover, we provide evidence that some complete graphs actually have the same matroid and thus this technique could never be used to distinguish the underlying statistical models.

Our first conjecture is purely graph-theoretic and is a direct generalization of Theorem \ref{thm:difGraphsPC_Set}. A proof of the first conjecture combined with Theorem \ref{outd_layer} would be one step toward a proof of the second conjecture.  

\begin{conjecture}
\label{conj:difGraphPCSets}
Let $G_1 = (V, D_1)$ and $G_2 = (D, V_2)$ be non-complete, simple directed graphs with the same out-degree sequence but which do not have the same strongly connected components. Then there exists a node $i \in G_k$ and a set $L$ which is parentally closed with respect to $i$ such that $|\ch_{k}(i)\cap L|>|\ch_{3-k}(i)\cap L|$. 
\end{conjecture}

\begin{conjecture}
Let $\mathcal{G}$ be any family of non-complete graphs, then graph parameter $G$ of the family $\{M_G\}_{G \in \mathcal{G}}$ of homoscedastic linear structural equation models is generically identifiable. 
\end{conjecture}

We have certified that Conjecture \ref{conj:difGraphPCSets} is true for all graphs with 4 and 5 nodes. In fact, for these cases it is true without the assumption that they have the same strongly connected components. For graphs with 6 nodes, we have tested 100,000 pairs of graphs for each possible out-degree sequence (or all possible pairs if there are less than 100,000). In each case Conjecture \ref{conj:difGraphPCSets} holds. The code for this is written in \texttt{SageMath} and can be found at:

\begin{center}
        \url{https://mathrepo.mis.mpg.de/cyclic-sem-identifiability}.
\end{center}

While Conjecture \ref{conj:difGraphPCSets} holds for all pairs of 4 and 5 node graphs, the assumption that the graphs do not have the exact same strongly connected components is necessary for larger graphs. This next example demonstrates this. It would be interesting to find a new graphical criterion which could be used to distinguish the matroids in this case. 

\begin{example}
Let $G_1$ and $G_2$ be the graphs pictured in Figure \ref{fig:notDistByPC}. Observe that they both have the same out-degree sequence which is $(3,2,2,2,2,0)$ and the same strongly connected components which consist of the vertices $\{1,2,3,4,5\}$ and $\{6\}$. 

With respect to each node, they both have the same nontrivial parentally closed sets which are
\begin{align*}
&1: \{6\}, \{2,3,4,5\}, &2: \{1,3,4,5\}, & &3: \{1,2,4,5\}, \\
&4: \{1,2,3,5\},  &5: \{1,2,3,4\}, & &6: \{1\}.
\end{align*}
One can see that for each node $i$ and parentally closed set $L$ above, the cardinality of the intersection $|\ch_k(i) \cap L|$ is the same. However, one can verify by direct computation that their associated Jacobian matroids are different and thus the models are still distinguishable. 
\end{example}

\begin{figure}
\centering
\begin{minipage}[t]{0.4\textwidth}
\centering
\begin{tikzpicture}
[main_node/.style={circle,fill=white,draw,minimum size=1 em,inner sep=2pt]}]

\node[main_node] (4) at (0, 0) {$4$};
\node[main_node] (5) at (5/2, 0) {$5$};
\node[main_node] (3) at (6.54/2, 4.75/2) {$3$};
\node[main_node] (2) at (2.50/2, 7.69/2) {$2$};
\node[main_node] (1) at (-1.54/2, 4.75/2) {$1$};
\node[main_node] (6) at (-1.54/2, 4.75/2 + 1) {$6$};

\draw [->, thick] (1) to (2);
\draw [->, thick] (1) to (4);
\draw [->, thick] (1) to (6);
\draw [->, thick] (2) to (4);
\draw [->, thick] (2) to (5);
\draw [->, thick] (3) to (1);
\draw [->, thick] (3) to (2);
\draw [->, thick] (4) to (3);
\draw [->, thick] (4) to (5);
\draw [->, thick] (5) to (1);
\draw [->, thick] (5) to (3);
\end{tikzpicture}
\end{minipage}
\begin{minipage}[t]{0.4\textwidth}
\centering
\begin{tikzpicture}
[main_node/.style={circle,fill=white,draw,minimum size=1 em,inner sep=2pt]}]

\node[main_node] (4) at (0, 0) {$4$};
\node[main_node] (5) at (5/2, 0) {$5$};
\node[main_node] (3) at (6.54/2, 4.75/2) {$3$};
\node[main_node] (2) at (2.50/2, 7.69/2) {$2$};
\node[main_node] (1) at (-1.54/2, 4.75/2) {$1$};
\node[main_node] (6) at (-1.54/2, 4.75/2 + 1) {$6$};

\draw [->, thick] (1) to (3);
\draw [->, thick] (1) to (5);
\draw [->, thick] (1) to (6);
\draw [->, thick] (2) to (1);
\draw [->, thick] (2) to (5);
\draw [->, thick] (3) to (2);
\draw [->, thick] (3) to (4);
\draw [->, thick] (4) to (1);
\draw [->, thick] (4) to (2);
\draw [->, thick] (5) to (3);
\draw [->, thick] (5) to (4);

\end{tikzpicture}
\end{minipage}
\caption{Two graphs $G_1$ and $G_2$ have the same out-degree sequence and same strongly connected components. There is no parentally closed set which can be used to distinguish them however they do have different matroids. }
\label{fig:notDistByPC}
\end{figure}
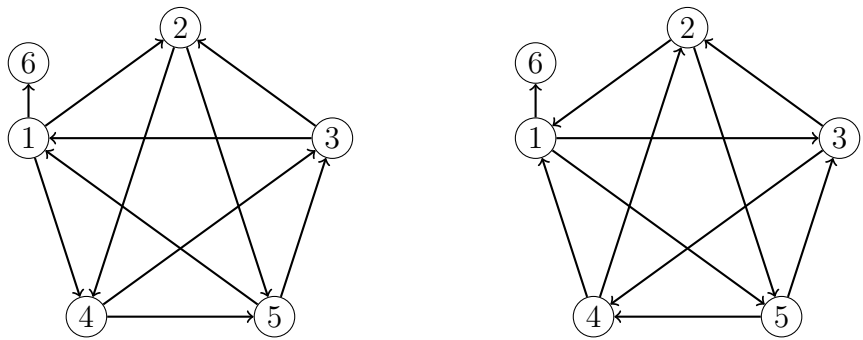

One assumption that all of our theorems and even our previous conjectures make is that the graphs involved are not complete. This is because many complete graphs actually have the same matroid, even when the graphs involved are DAGs. We end with our conjecture on when this happens for our family of models. 

\begin{conjecture}
Let $G_1$ and $G_2$ be complete digraphs on $p$ nodes with all of the same edge directions except on the edge between $p-1$ and $p$. Then $G_1$ and $G_2$ have the same Jacobian matroid. 
\end{conjecture}

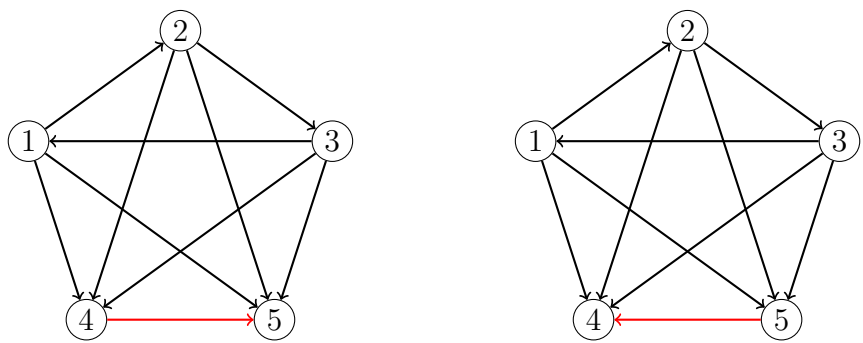
\begin{figure}
\centering
\begin{minipage}[t]{0.4\textwidth}
\centering
\begin{tikzpicture}
[main_node/.style={circle,fill=white,draw,minimum size=1 em,inner sep=2pt]}]

\node[main_node] (4) at (0, 0) {$4$};
\node[main_node] (5) at (5/2, 0) {$5$};
\node[main_node] (3) at (6.54/2, 4.75/2) {$3$};
\node[main_node] (2) at (2.50/2, 7.69/2) {$2$};
\node[main_node] (1) at (-1.54/2, 4.75/2) {$1$};

\draw [->, thick] (1) to (2);
\draw [->, thick] (2) to (3);
\draw [->, thick] (3) to (1);
\draw [->, thick] (1) to (4);
\draw [->, thick] (2) to (4);
\draw [->, thick] (3) to (4);
\draw [->, thick] (1) to (5);
\draw [->, thick] (2) to (5);
\draw [->, thick] (3) to (5);
\draw [->, thick, red] (4) to (5);

\end{tikzpicture}
\end{minipage}
\begin{minipage}[t]{0.4\textwidth}
\centering
\begin{tikzpicture}
[main_node/.style={circle,fill=white,draw,minimum size=1 em,inner sep=2pt]}]

\node[main_node] (4) at (0, 0) {$4$};
\node[main_node] (5) at (5/2, 0) {$5$};
\node[main_node] (3) at (6.54/2, 4.75/2) {$3$};
\node[main_node] (2) at (2.50/2, 7.69/2) {$2$};
\node[main_node] (1) at (-1.54/2, 4.75/2) {$1$};

\draw [->, thick] (1) to (2);
\draw [->, thick] (2) to (3);
\draw [->, thick] (3) to (1);
\draw [->, thick] (1) to (4);
\draw [->, thick] (2) to (4);
\draw [->, thick] (3) to (4);
\draw [->, thick] (1) to (5);
\draw [->, thick] (2) to (5);
\draw [->, thick] (3) to (5);
\draw [->, thick, red] (5) to (4);

\end{tikzpicture}
\end{minipage}
\caption{Two complete digraphs with the same Jacobian matroid.} 
\label{fig5}
\end{figure}

This conjecture means that while many complete graphs might be distinguishable from each other, this matroid-based technique is unable to prove this. This is similar to Conjecture 4.8 of \cite{Seth19M} where the authors found a pair of phylogenetic models which they conjecture have the same matroid despite the statistical models (or their associated algebraic varieties) actually being different. In general, this can happen because two different algebraic varieties can define the exact same algebraic matroid. We also have verified this conjecture for all complete digraphs with up to 6 nodes using \texttt{Macaulay2}. The code for this along with an explanation of how all of our code works can be found in our MathRepo page \cite{mathrepo}. 

\bibliography{references.bib}{}
\bibliographystyle{plain}
\end{document}